\documentclass[reqno,11pt]{amsart}
\usepackage[margin=1in]{geometry}

\usepackage{amsthm, amsmath, amssymb, bm}
\usepackage{tikz}

\usepackage{microtype} 

\usepackage[utf8]{inputenc}
\usepackage[T1]{fontenc}

\usepackage[textsize=small,backgroundcolor=orange!20]{todonotes}

\usepackage[hidelinks]{hyperref}
\usepackage{url}

\usepackage{xcolor,graphics}

\newtheorem{theorem}{Theorem}[section]

\newtheorem{lemma}[theorem]{Lemma}

\newtheorem{conjecture}[theorem]{Conjecture}
\newtheorem*{question*}{Question}

\theoremstyle{definition}

\theoremstyle{remark}
\newtheorem*{remark}{Remark}

\DeclareMathOperator{\iso}{iso}

\tikzstyle{P} = [draw, circle, black, fill, inner sep = 0pt, minimum width = 3pt]
\tikzstyle{every loop} = []
\newcommand{\tikzHind}{
  \begin{tikzpicture}[baseline, yshift=1pt]
    \path[use as bounding box] (-.15,-.1) rectangle (.6,.35);
    \draw (0.5,0) node[P] {} -- (0,0) node[P] {} edge[-,in = 45, out = 135, loop] ();
  \end{tikzpicture}
}

\title{The number of independent sets in an irregular graph}

\author[Sah]{Ashwin Sah}
\address{Massachusetts Institute of Technology, Cambridge, MA 02139, USA}
\email{asah@mit.edu}

\author[Sawhney]{Mehtaab Sawhney}
\address{Massachusetts Institute of Technology, Cambridge, MA 02139, USA}
\email{msawhney@mit.edu}

\author[Stoner]{David Stoner}
\address{Harvard University, Cambridge, MA 02138, USA}
\email{dstoner@college.harvard.edu}

\author[Zhao]{Yufei Zhao}
\address{Department of Mathematics, Massachusetts Institute of Technology, Cambridge, MA 02139, USA}
\email{yufeiz@mit.edu}
\thanks{YZ was supported by NSF Award DMS-1362326.}

\date{May 2018 (initial). Jan 2019 (revised)}

\begin{document}

\begin{abstract}
Settling Kahn's conjecture (2001), we prove the following upper bound on the number $i(G)$ of independent sets in a graph $G$ without isolated vertices:
\[
i(G) \le \prod_{uv \in E(G)} i(K_{d_u,d_v})^{1/(d_u d_v)},
\]
where $d_u$ is the degree of vertex $u$ in $G$. Equality occurs when $G$ is a disjoint union of complete bipartite graphs. The inequality was previously proved for regular graphs by Kahn and Zhao. 

We also prove an analogous tight lower bound:
\[
i(G) \ge \prod_{v \in V(G)} i(K_{d_v+1})^{1/(d_v + 1)},
\]
where equality occurs for $G$ a disjoint union of cliques. More generally, we prove bounds on the weighted versions of these quantities, i.e., the independent set polynomial, or equivalently the partition function of the hard-core model with a given fugacity on a graph.
\end{abstract}

\maketitle

\section{Introduction}

Among $d$-regular graphs on $n$ vertices, which one has the most number of independent sets? This question was initially raised by Granville in connection with problems from combinatorial number theory. It was conjectured by Alon~\cite{Alon91} and Kahn~\cite{Kahn01} that, when $n$ is divisible by $2d$, the $n$-vertex $d$-regular graph with the maximum number of independent sets is a disjoint union of complete bipartite graph $K_{d,d}$'s. The conjecture was proved by Kahn~\cite{Kahn01} for bipartite graphs using a beautiful entropy argument, and extended to all regular graphs by Zhao~\cite{Zhao10} via a combinatorial reduction to the bipartite case. Specifically, the following theorem was shown. We write $i(G)$ for the number of independent sets of a graph $G$.

\begin{theorem}[Kahn~\cite{Kahn01}, Zhao~\cite{Zhao10}] \label{thm:kahn-zhao}
	Let $G$ be an $n$-vertex $d$-regular graph. Then
	\[
	i(G) \le i(K_{d,d})^{n/(2d)} = (2^{d+1} - 1)^{n/(2d)}.
	\]
	Equality holds if and only if $G$ is a disjoint union of $K_{d,d}$'s.
\end{theorem} 

Note that $i(G \sqcup H) = i(G) i(H)$, where $G \sqcup H$ denotes a disjoint union of two graphs. If we exponentially normalize the number of independent sets as $i(G)^{1/|V(G)|}$, then the theorem says that among $d$-regular graphs, this quantity is maximized by $G = K_{d,d}$, as well as disjoint unions of copies of $K_{d,d}$.

As many interesting combinatorial problems can be phrased in terms of independent sets in graphs and hypergraphs, the problem of bounding the number of independent sets is of central interest. For example, see the ICM 2018 survey~\cite{BMS18} on the recent breakthroughs on the hypergraph container method of Balogh, Morris, and Samotij~\cite{BMS15} and independently Saxton and Thomason~\cite{ST15}, which built partly on the earlier work by Sapozhenko~\cite{Sap01}, a precursor to Theorem~\ref{thm:kahn-zhao}, giving a weaker upper bound for $i(G)$.

Recently, Davies,  Jenssen, Perkins, and Roberts~\cite{DJPR1} proved a strengthening of Theorem~\ref{thm:kahn-zhao} using a novel technique they called the ``occupancy method'', which has also been applied to other settings such as matchings, colorings, and Euclidean sphere packings~\cite{DJPR1,DJPR2,DJPR3,JJP1,JJP2}. See the recent survey~\cite{Zhao17} for an overview of related developments.

Kahn~\cite{Kahn01} conjectured an extension of Theorem~\ref{thm:kahn-zhao} to not necessarily regular graphs, where the conjectured maximizer is also a disjoint union of complete bipartite graphs $K_{a,b}$'s, where $a,b$ may differ for each component. Specifically, it was conjectured that for a graph $G$ without isolated vertices (i.e., degree-0 vertices),
\[
i(G) \le \prod_{uv \in E(G)} i(K_{d_u,d_v})^{1/(d_u d_v)},
\]
where $d_u$ is the degree of vertex $u$ in $G$.

The conjecture can be rephrased in terms of the following extremal problem. Let the \emph{degree-degree distribution} of $G$ be the probability distribution of the unordered pair $\{d_u, d_v\}$ as $uv$ ranges uniformly over edges of $G$. An example of a degree-degree distribution is that $20\%$ of edges have one endpoint having degree 2 and the other degree 3, $30\%$ of edges have $(3,3)$, and $50\%$ of edges have $(3,4)$. What the maximum of $i(G)^{1/v(G)}$ over all graphs $G$ with a given degree-degree distribution? Kahn's conjecture states that the maximum is attained when $G$ is a disjoint union of complete bipartite graphs with the prescribed degree-degree distribution of edges.

Galvin and Zhao~\cite{GZ11} gave a computer-assisted proof of the conjecture when the maximum degree of $G$ is at most 5. It is not known if the recent occupancy method~\cite{DJPR1} can be extended to irregular graphs, as there appear to be some fundamental obstacles.

Our main result, below, proves Kahn's conjecture, thereby generalizing Theorem~\ref{thm:kahn-zhao} to irregular graphs. 

\begin{theorem} \label{thm:upper}
	Let $G$ be a graph without isolated vertices. Let $d_v$ the degree of vertex $v$ in $G$. Then
	\[
	i(G) \le \prod_{uv \in E(G)} i(K_{d_u,d_v})^{1/(d_u d_v)} 
	= \prod_{uv \in E(G)} (2^{d_u} + 2^{d_v} - 1)^{1/(d_ud_v)}.
	\]
	Equality holds if and only if $G$ is a disjoint union of complete bipartite graphs.
\end{theorem}

\begin{remark}
A vertex version of this inequality, i.e., $i(G) \le \prod_{v \in V(G)} i(K_{d_v,d_v})^{1/(2d_v)}$, is false, e.g., for a path on 4 vertices, as $8 \not\le \sqrt{63}$.
\end{remark}

Kahn's proof~\cite{Kahn01} of the bipartite case of Theorem~\ref{thm:kahn-zhao} made clever use of Shearer's entropy inequality~\cite{CFGS86}. It remains unclear how to apply Shearer's inequality in a lossless way in the irregular case, despite previous attempts to do so, e.g., \cite{MT10}. Kahn's entropy proof was later generalized to the weighted setting (see \eqref{eq:upper-reg-weighted} below) by Galvin and Tetali~\cite{GT04}, as well as more generally to graph homomorphisms (also see \cite{Gal14}), though the entropy proof remained the only approach known until Lubetzky and Zhao~\cite{LZ14} gave a ``one-line'' proof via H\"older's inequality, which can be viewed as a re-interpretation of Kahn's entropy proof (see \cite{Fri04} for a discussion relating Shearer's inequality to H\"older's inequality). Still, the H\"older's inequality method in \cite{LZ14} could not handle irregular graphs. Our new result in this paper hints at the possibility of a powerful new ``non-uniform H\"older's inequality'' that could have much wider applications, though we do not speculate here on the exact form of such a more general inequality.

\bigskip

We also prove an analogous but somewhat easier lower bound. The number of independent sets, exponentially normalized as $i(G)^{1/|V(G)|}$, is known to be minimized among $d$-regular graphs by $G = K_{d+1}$.

\begin{theorem}[Cutler and Radcliffe~\cite{CR14}] \label{thm:lower-reg}
	Let $G$ be an $n$-vertex $d$-regular graph. Then
	\[
		i(G) \ge i(K_{d+1})^{n/(d+1)} = (d+2)^{n/(d+1)}.
	\]
	Equality holds if and only if $G$ is a disjoint union of $K_{d+1}$'s.
\end{theorem}

Our second result extends the above inequality to irregular graphs.

\begin{theorem} \label{thm:lower}
	Let $G$ be graph and $d_v$ the degree of vertex $v$ in $G$. Then
	\[i(G)\ge \prod_{v\in V(G)}(d_v+2)^{1/(d_v+1)}.\]
	Equality holds if and only if $G$ is a disjoint union of cliques.
\end{theorem}

We also establish weighted versions of the above results. Let the \emph{independent set polynomial} of $G$ be
\[
P_G(\lambda) = \sum_{I\in \mathcal{I}(G)} \lambda^{|I|}.
\]
Here $ \mathcal{I}(G)$ denotes the set of independent sets of $G$. Note that $P_G(1) = i(G)$. This polynomial is the weighted sum over all independent sets $I$ of $G$, where the set $I$ is assigned weight $\lambda^{|I|}$. The parameter $\lambda$ is usually called \emph{fugacity}. The quantity $P_G(\lambda)$ is the partition function of the \emph{hard-core model with fugacity $\lambda$} from statistical physics, which is an important model for choosing a random independent set $I$ of $G$, where each $I$ is chosen with probability proportional to $\lambda^{|I|}$.

Theorem~\ref{thm:kahn-zhao} was extended by Galvin and Tetali~\cite{GT04} (along with the same reduction by Zhao \cite{Zhao10}) to $P_G(\lambda)$, showing that for every $n$-vertex $d$-graph graph $G$, and parameter $\lambda > 0$, we have
\begin{equation} \label{eq:upper-reg-weighted}
P_G(\lambda) \le P_{K_{d,d}}(\lambda)^{n/(2d)} = (2(1+\lambda)^d - 1)^{n/(2d)}.
\end{equation}
We extend this result to irregular graphs. Theorem~\ref{thm:upper} is the $\lambda=1$ special case of the following result.
\begin{theorem} \label{thm:upper-weighted}
	Let $G$ be graph without isolated vertices. Let $d_v$ the degree of vertex $v$ in $G$. Let $\lambda > 0$. Then
	\[
		P_G(\lambda) \le \prod_{uv \in E(G)} P_{K_{d_v,d_u}}(\lambda)^{1/(d_ud_v)} 
		= \prod_{uv \in E(G)} ((1+\lambda)^{d_u} + (1+\lambda)^{d_v}-1)^{1/(d_ud_v)}.
	\]
	Equality holds if and only if $G$ is a disjoint union of complete bipartite graphs.
\end{theorem}

Theorem~\ref{thm:upper-weighted} reduces to bipartite $G$ via \cite{Zhao10} as we will explain in Section~\ref{sec:strategy}. For bipartite graphs, we have the following slightly more general result that allows two different weights.

A \emph{bigraph} $G = (A, B, E)$ is a bipartite graph with a specified vertex bipartition $V(G) = A \sqcup B$ and edge set $E \subseteq A \times B$. We define the two-variable independent set polynomial of the bigraph $G$ by
\[
P_G(\lambda, \mu) = \sum_{I\in \mathcal{I}(G)} \lambda^{|I \cap A|} \mu^{|I \cap B|}.
\]
Theorem~\ref{thm:upper-weighted} has the following bivariate extension.

\begin{theorem} \label{thm:upper-biweighted}
	Let $G = (A, B, E)$ be a bigraph without isolated vertices. Let $d_v$ denote the degree of vertex $v$ in $G$. Let $\lambda, \mu > 0$. Then
	\[
		P_G(\lambda,\mu) \le \prod_{\substack{uv \in E \\ u \in A, v\in B}} ((1+\lambda)^{d_v} + (1+\mu)^{d_u}-1)^{1/(d_ud_v)}.
	\]
	Equality holds if and only if $G$ is a disjoint union of complete bipartite graphs.
\end{theorem}

We also generalize the lower bound Theorem~\ref{thm:lower} to the independent set polynomial. Theorem~\ref{thm:lower} follows from the next result by setting $\lambda = 1$.

\begin{theorem} \label{thm:lower-weighted}
Let $G$ be a graph. Let $d_v$ denote the degree of vertex $v$ in $G$. Let $\lambda > 0$. Then
\[
P_G(\lambda) \ge \prod_{v \in V(G)} P_{K_{d_v+1}}(\lambda)^{1/(d_v+1)} = \prod_{v \in V(G)} ((d_v+1)\lambda + 1)^{1/(d_v+1)}.
\]
Equality holds if and only if $G$ is a disjoint union of cliques.
\end{theorem}

The proofs of all these theorems follow an induction strategy used by Galvin and Zhao \cite{GZ11}, which we outline in the next section. In \cite{GZ11} the strategy was carried out to prove the upper bound, Theorem~\ref{thm:upper}, for graphs of maximum degree at most 5 with the help of a computer. In this paper, we establish a number of analytic inequalities that allow us to prove the results without the maximum degree assumption. The proofs of some of these inequalities are fairly technical verifications, and they are deferred to the appendix.

After outlining the strategy, we prove the lower bound results, Theorems~\ref{thm:lower} and \ref{thm:lower-weighted}, in Section~\ref{sec:lower}, followed by the upper bound results, Theorem~\ref{thm:upper}, \ref{thm:upper-weighted}, and \ref{thm:upper-biweighted}, in Section~\ref{sec:upper}. Both proofs use similar ideas, but the upper bound proof is more challenging to execute.

Finally, we conclude in Section~\ref{sec:concluding} by offering some corollaries, including how to bound the number of independent sets given the degree distribution of a graph. We also give some remarks on potential applications of the method to other open problems, such as counting the number of colorings and graph homomorphisms.

\section{Proof strategy} \label{sec:strategy}

The proof proceeds by induction on the number of vertices of $G$. Let us sketch the proof of the upper bound in the unweighted setting (Theorem~\ref{thm:upper}). The strategy for the lower bound (Theorem~\ref{thm:lower}) is similar.

Let $\iso(G)$ denote the number of isolated vertices in $G$. Set
\[
j(G) := 2^{\iso(G)} \prod_{uv \in E(G)} i(K_{d_u,d_v})^{1/(d_ud_v)}.
\]
Theorem~\ref{thm:upper} then says that $i(G) \le j(G)$ for all graphs $G$. 

In~\cite{Zhao10}, Theorem~\ref{thm:kahn-zhao}, the upper bound on the number of independent sets in a regular graph, was reduced to bipartite graphs via a \emph{bipartite swapping trick} (later elaborated in~\cite{Zhao11}). It was shown that $i(G)^2 \le i(G \times K_2)$. Here $\times$ denotes the graph tensor product. The graph $G \times K_2$ is also known as the \emph{bipartite double cover} of $G$, and it has vertices $V(G) \times \{0,1\}$, and an edge between $(u,0)$ and $(v,1)$ for every $uv \in E(G)$. It is easy to see that $j(G)^2 = j(G \times K_2)$, since lifting $G$ to its bipartite double cover $G \times K_2$ preserves degrees. Thus it suffices to show that $i(G \times K_2) \le  j(G \times K_2)$, which reduces to proving $i(G) \le j(G)$ for all bipartite graphs $G$.

We use induction on the number of vertices of $G$. Also, since both $i(G)$ and $j(G)$ factor over connected components of $G$, we may assume that $G$ is connected.

The number of independent sets $i(G)$ satisfies the following easy recurrence relation. For every vertex $w$,
\[
i(G) = i(G - w) + i(G - w - N(w)),
\]
where $G-w$ denotes $G$ with the vertex $w$ deleted (along with all edges incident to $w$), and $G - w - N(w)$ denotes $G$ with $w$ and all neighbors of $w$ deleted. The recurrence relation follows from noting that $i(G-w)$ counts the number of independent sets of $G$ not containing $w$, and $i(G - w - N(w))$ counts the number of independent sets of $G$ containing $w$. Applying induction, it suffices to show that, for if $w$ is a maximum degree vertex of $G$, then
\begin{equation} \label{eq:upper-recursive}
j(G - w) + j(G - w - N(w)) \le j(G).
\end{equation}
This inequality was conjectured by Galvin and Zhao~\cite{GZ11}, with a computer-assisted proof\footnote{Ad-hoc tricks were used in \cite{GZ11} to handle maximum degree 5 graphs, due to computational limitations.} when $G$ has maximum degree at most 4. Here we prove the above inequality for all $G$ and an arbitrary maximum degree vertex $w$.

Let $V_k$ denote the the set of vertices at distance exactly $k$ from the vertex $w$. So in particular $V_0 = \{w\}$ and $V_1 = N(w)$. Since $G$ is assumed bipartite, there are no edges within each $V_k$. Let $E_k$ denote the edges between $V_{k-1}$ and $V_k$. Write $E_{\ge k} := \bigcup_{i \ge k} E_i$. We have $E = E_{\ge 1}$ since $G$ is connected. See Figure~\ref{fig:setup}.

\begin{figure}
\centering
\includegraphics{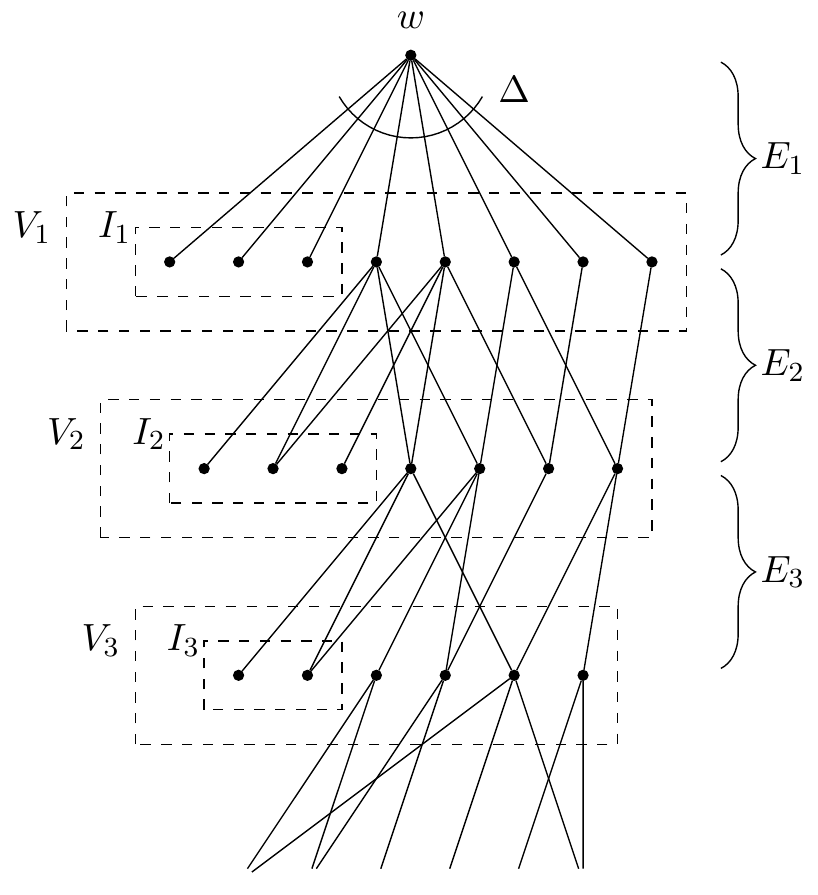}
\caption{Setup for the proofs of Theorems 1.2 and 1.4} \label{fig:setup}
\end{figure}

For each $v \in V_k$, its neighbors are contained in $V_{k-1} \cup V_{k+1}$. We write $d^+_v$ to denote the number of neighbors of $v$ contained in $V_{k+1}$. Then the terms in \eqref{eq:upper-recursive} can be written as
\begin{align*}
j(G) &= 2^{\iso(G)} \prod_{(u,v) \in E} i(K_{d_u,d_v})^{1/(d_ud_v)},
\\
j(G-w) &= 2^{\iso(G-w)}\prod_{\substack{(u,v) \in E_2 \\  v \in V_1}} i(K_{d_u,d_v^+})^{1/(d_ud_v^+)}
 \prod_{(u,v) \in E_{\ge 3}} i(K_{d_u,d_v})^{1/(d_ud_v)}, \text{ and}
\\
j(G-w - N(w)) &= 2^{\iso(G-w-N(w))}\prod_{\substack{(u,v) \in E_3 \\  u \in V_2}} i(K_{d_u^+,d_v})^{1/(d_u^+d_v)}
 \prod_{(u,v) \in E_{\ge 4}} i(K_{d_u,d_v})^{1/(d_ud_v)}.
\end{align*}
Observe that the factor $\prod_{k \ge 4} \prod_{(u,v) \in E_k} i(K_{d_u,d_v})^{1/(d_ud_v)}$ is present is all three expressions. By eliminating this common factor, we see that \eqref{eq:upper-recursive} reduces to
\begin{multline*}
2^{|I_1|}\prod_{\substack{(u,v) \in E_2 \\  v \in V_1}} i(K_{d_u,d_v^+})^{1/(d_ud_v^+)}
\prod_{(u,v) \in E_3} i(K_{d_u,d_v})^{1/(d_ud_v)}
+
2^{|I_2|} \prod_{\substack{(u,v) \in E_3 \\  u \in V_2}} i(K_{d_u^+,d_v})^{1/(d_u^+d_v)}
\\
\le \prod_{(u,v) \in E_{\le 3}} i(K_{d_u,d_v})^{1/(d_ud_v)},
\end{multline*}
where $I_k$ is the set vertices in $V_k$ that become isolated once we delete $V_{k-1}$ from $G$. In other words, $V_k$ is the set of vertices whose neighborhood is contained in $V_{k-1}$. Thus the inequality \eqref{eq:upper-recursive} only depends on the subgraph of $G$ induced by $V_0 \cup V_1 \cup V_2 \cup V_3$, which is a more tractable problem.\footnote{This is in fact a finite problem if we fix the maximum degree of $G$. This observation led to the approach in \cite{GZ11}.} We prove the above inequality by carefully analyzing the quantities $i(K_{a,b})^{1/(ab)}$, including some judicious applications of H\"older's inequality.

\section{Lower bound} \label{sec:lower}

In this section we prove Theorem~\ref{thm:lower-weighted}, which states that $P_G(\lambda) \ge P_G^-(\lambda)$, where recall $P_G(\lambda) = \sum_{I\in\mathcal{I}(G)} \lambda^{|I|}$, and we set
\[
P_G^-(\lambda) := \prod_{v \in V(G)} ((d_v+1)\lambda + 1)^{1/(d_v+1)}.
\]

We proceed by induction on the number of vertices in $G$. The case $|V(G)| = 1$ is trivial. Since $P_{G \sqcup H} (\lambda) = P_G(\lambda)P_H(\lambda)$ and $P_{G \sqcup H}^- (\lambda) = P_G^-(\lambda)P_H^-(\lambda)$, it suffices to prove the inequality when $G$ is connected. 

Suppose $G$ has maximum degree $\Delta$. Let $w$ be a vertex of degree $\Delta$. Let $V_k$ denote the set of vertices at distance exactly $k$ from $w$, e.g., $V_0 = \{w\}$ and $V_1 = N(w)$. Write $V_{\ge k} = \bigcup_{i \ge k} V_i$. Since $G$ is connected, $V(G) = V_{\ge 0}$. For $u\in V_2$, let $e_u$ be the number of its neighbors in $V_1 = N(w)$, and $f_u$ be the number of remaining neighbors, so that $e_u + f_u = d_u$. 

By considering independent sets containing $w$ versus those that do not, we obtain the recursion
\[
P_G(\lambda) = P_{G - w}(\lambda) + \lambda P_{G - w - N(w)}(\lambda).
\]
By the induction hypothesis, we have $P_{G - w}(\lambda) \ge P_{G - w}^{-}(\lambda)$ and $P_{G - w - N(w)}(\lambda) \ge P_{G - w -N(w)}^{-}(\lambda)$. Thus to prove $P_G(\lambda) \ge P_G^-(\lambda)$ it suffices to show
\begin{equation} \label{eq:lower-P-recur}
	P_{G - w}^-(\lambda) + \lambda P_{G - w - N(w)}^-(\lambda) \ge P_G^-(\lambda).
\end{equation}
We have
\[
P_{G-w}^-(\lambda) = \prod_{v\in V_1} (d_v\lambda+1)^{\frac{1}{d_v}}  \prod_{v\in V_{\ge 2}} ((d_v+1)\lambda+1)^{\frac{1}{d_v + 1}}
\]
and
\[
\lambda P_{G-w-N(w)}(\lambda) = \lambda \prod_{v\in V_2} ((f_v+1)\lambda+1)^{\frac{1}{f_v + 1}} \prod_{v\in V_{\ge 3}} ((d_v+1)\lambda+1)^{\frac{1}{d_v + 1}}.
\]
After removing the common the factor $\prod_{v\in V_{\ge 3}} ((d_v+1)\lambda+1)^{\frac{1}{d_v + 1}}$, \eqref{eq:lower-P-recur} is seen to be equivalent to
\[
\prod_{v\in V_1} (d_v\lambda+1)^{\frac{1}{d_v}} \prod_{v\in V_2} ((d_v+1)\lambda+1)^{\frac{1}{d_v + 1}} +\lambda  \prod_{v\in V_2} ((f_v+1)\lambda+1)^{\frac{1}{f_v + 1}}
\ge \prod_{v\in V_0\cup V_1\cup V_2} ((d_v+1)\lambda+1)^{\frac{1}{d_v + 1}}.
\]
On the right-hand side, the only $V_0$ contribution is $v = w$ with $d_w = \Delta$. Dividing both sides by the $V_2$ contributions, we see that the inequality is equivalent to
\begin{equation} \label{eq:lower-expand-divide}
\prod_{v\in V_1}(d_v\lambda+1)^{\frac{1}{d_v}} + \lambda \prod_{v\in V_2}\frac{((f_v+1)\lambda+1)^\frac{1}{f_v + 1}}{((d_v+1)\lambda+1)^{\frac{1}{d_v + 1}}} \ge ((\Delta+1)\lambda+1)^{\frac{1}{\Delta+ 1}}\prod_{v\in V_1} (d_v\lambda+1)^{\frac{1}{d_v}}.
\end{equation}
Observe that 
\begin{equation}
	\label{eq:lower-exp-ineq}
	(a + 1)^{1/a} > (b + 1)^{1/b} \text{ for } 0 < a < b,
\end{equation}
which follows from taking logarithms and noting that $\log (x + 1)$ is concave, so that $\log(x+1)/x$ is decreasing for $x > 0$. So $((f_v+1)\lambda+1)^\frac{1}{f_v + 1} \ge ((d_v+1)\lambda+1)^{\frac{1}{d_v + 1}}$ as $f_v\le d_v$. Thus, to prove \eqref{eq:lower-expand-divide}, it suffices to prove that
\begin{equation}
	\label{eq:lower-expand-reduce}
	\prod_{v\in V_1} (d_v \lambda + 1)^{\frac{1}{d_v}} + \lambda\ge ((\Delta + 1)\lambda + 1)^{\frac{1}{\Delta + 1}}\prod_{v\in V_1} ((d_v + 1)\lambda +1 )^{\frac{1}{d_v + 1}}.
\end{equation}
In fact, we will prove this inequality for arbitrary reals $d_v \in [1,\Delta]$ for $v \in V_1$. Recall that $|V_1| = |N(w)| = \Delta$. Let
\[
	f(d_1, \dots, d_\Delta) := \frac{\prod_{v=1}^\Delta (d_v \lambda+ 1)^{\frac{1}{d_v}} + \lambda}
{\prod_{v = 1}^\Delta ((d_v + 1)\lambda+1)^{\frac{1}{d_v + 1}}}.
\]
It suffices to show that $f(d_1,\dots, d_\Delta) \ge f(\Delta, \dots, \Delta) = ((\Delta + 1)\lambda +1)^{\frac{1}{\Delta + 1}}$ for all reals $d_1, \dots, d_\Delta \in [1,\Delta]$. 

Since $f$ is symmetric in its variables, it suffices to show $f(d_1, \dots, d_\Delta) \ge f(\Delta, d_2, \dots, d_\Delta)$ for all reals $d_2, \dots, d_\Delta \in [1,\Delta]$, so that we can iterate and replace each variable $d_v$ by $\Delta$.

By \eqref{eq:lower-exp-ineq}, we have $(d_{v} \lambda + 1)^{\frac{1}{d_{v}}}\ge (\Delta \lambda+ 1)^{\frac{1}{\Delta}}$ for each $v$. Using the fact that $a\ge b \ge 0$ and $c\ge d \ge 0$ imply $\frac{ac + 1}{bc + 1}\ge\frac{ad + 1}{bd + 1}$ (it is equivalent to $(a - b)(c - d)\ge 0$), we have
\begin{align*}
\frac{f(d_1, d_2, \dots, d_\Delta)}{f(\Delta, d_2, \dots, d_\Delta)}
&= \left( 
	\frac{
	  (d_1\lambda+1)^{\frac{1}{d_1}} 
      \prod_{v =2}^\Delta (d_v\lambda+1)^{\frac{1}{d_v}} +\lambda
     }
     {
      (\Delta\lambda+1)^{\frac{1}{\Delta}} 
      \prod_{v =2}^\Delta (d_v\lambda+1)^{\frac{1}{d_v}} +\lambda
     }
   \right)
   \left(
    \frac{((\Delta+1)\lambda+1)^{\frac{1}{\Delta+1}}}{((d_1+1)\lambda+1)^{\frac{1}{d_1+1}}}
   \right)
\\
&\ge \left( 
	\frac{
	  (d_1\lambda+1)^{\frac{1}{d_1}} 
      \prod_{v =2}^\Delta (\Delta\lambda+1)^{\frac{1}{\Delta}} +\lambda
     }
     {
      (\Delta\lambda+1)^{\frac{1}{\Delta}} 
      \prod_{v =2}^\Delta (\Delta\lambda+1)^{\frac{1}{\Delta}} +\lambda
     }
   \right)
   \left(
    \frac{((\Delta+1)\lambda+1)^{\frac{1}{\Delta+1}}}{((d_1+1)\lambda+1)^{\frac{1}{d_1+1}}}
   \right)
\\ 
&= \frac{(d_1\lambda+1)^{\frac{1}{d_1}} (\Delta\lambda+1)^{\frac{\Delta-1}{\Delta}}+\lambda}{((\Delta+1)\lambda+1)^{\frac{\Delta}{\Delta+1}}((d_1+1)\lambda+1)^{\frac{1}{d_1+1}}}.
\end{align*}
Thus it remains to prove
\begin{equation}
	\label{eq:pf-lower-final}
	(\Delta\lambda+1)^\frac{\Delta-1}{\Delta}(d \lambda+1)^\frac{1}{d}+\lambda\ge ((\Delta+1)\lambda+1)^\frac{\Delta}{\Delta+1}((d+1)\lambda+1)^\frac{1}{d+1}
\end{equation}
for $1\le d \le\Delta$, which is proved in Lemma~\ref{lem:g-derivative}. 

\smallskip

\emph{Equality conditions.} Suppose equality occurs in Theorem~\ref{thm:lower-weighted}. We still assume that $G$ is connected. Since $\lambda > 0$, Lemma~\ref{lem:g-derivative} implies that to have equality in \eqref{eq:pf-lower-final}, we must have $d = \Delta$. Therefore equality holds in $f(d_1,\dots, d_\Delta) \ge f(\Delta, \dots, \Delta)$ if and only if $d_1 = \dots = d_\Delta = \Delta$. Thus $d_v = \Delta$ for all $v \in V_1$. Since \eqref{eq:lower-exp-ineq} is strict for $a< b$, to maintain equality in reducing \eqref{eq:lower-expand-divide} to \eqref{eq:lower-expand-reduce}, we must have $f_v=d_v$ for all $v\in V_2$, but this is impossible unless $V_2$ is empty since every vertex in $V_2$ is adjacent to some vertex in $V_1$. Therefore, $V_2$ is empty, which forces $G = K_{\Delta+1}$.

The inequality is strict for all connected $G$ except for cliques. Since the inequality factors over connected components, we see that equality occurs for a general graph $G$ if and only if  $G$ is a disjoint union of cliques. 
This completes the proof of Theorem~\ref{thm:lower-weighted}.

\section{Upper bound} \label{sec:upper}

In this section we prove Theorem~\ref{thm:upper-biweighted}. Note that Theorem~\ref{thm:upper-weighted} (and hence Theorem~\ref{thm:upper}) follows by setting $\mu = \lambda$ in Theorem~\ref{thm:upper-biweighted} and using $P_G(\lambda)^2 \le P_{G \times K_2} (\lambda)^2$ from \cite{Zhao10} to reduce Theorem~\ref{thm:upper-weighted} to the bipartite setting.

For a bigraph $G = (A, B, E)$, where $E \subseteq A \times B$, recall
$ P_G(\lambda, \mu) = \sum_{I\in\mathcal{I}(G)} \lambda^{|I \cap A|} \mu^{|I \cap B|}$.
Let $\iso_A(G)$ and $\iso_B(G)$ denote the number of isolated vertices of $G$ lying in $A$ and $B$ respectively. Set
\[
P_G^+(\lambda, \mu) := (1+\lambda)^{\iso_A(G)} (1+\mu)^{\iso_B(G)} \prod_{(u,v) \in E} ((1+\mu)^{d_u} + (1+\lambda)^{d_v} - 1 )^{1/(d_ud_v)}.
\]
We use the notation convention that $u \in A$ and $v \in B$ (this is consistent with $(u,v) \in E$ as $E \subseteq A \times B$ is a set of ordered pairs). Our aim is to prove Theorem~\ref{thm:upper-biweighted}, which says that $P_G(\lambda, \mu) \le P_G^+(\lambda, \mu)$ for all bigraphs $G$ and weights $\lambda, \mu > 0$.

We use induction on the number of vertices of $G$. If $G$ has maximum degree at most 1, i.e., a union of isolated edges and vertices, then the theorem is trivial to verify.

Since both $P_G(\lambda, \mu)$ and $P_G^+(\lambda,\mu)$ factor over connected components of $G$, we may assume that $G$ is connected.

Suppose $G$ has maximum degree $\Delta \ge 2$. Let $w$ be a vertex of degree $\Delta$. Without loss of generality, assume that $w \in A$. Let $V_k$ denote the set of vertices at distance exactly $k$ from $w$, e.g., $V_0 = \{w\}$ and $V_1 = N(w)$. Write $V_{\ge k} = \bigcup_{i \ge k} V_i$. Note that $V_{2k} \subseteq A$ and $V_{2k+1} \subseteq B$. For each $i \ge 1$, define $E_i \subseteq E$ to be the set of edges of the bigraph between $V_{i-1}$ and $V_i$. Write $E_{\ge k} = \bigcup_{i \ge k} E_i$. Since $G$ is a connected, $E = E_{\ge 1}$. By considering independent sets of $G$ containing $v$ and those not containing $v$, we have
\[
P_G(\lambda, \mu) = P_{G - w}(\lambda, \mu) + \lambda  P_{G - w - N(w)}(\lambda, \mu).
\]
By induction, it suffices to prove that
\begin{equation} \label{eq:upper-P-recur}
	P_G^+(\lambda, \mu) \ge P_{G - w}^+(\lambda, \mu) + \lambda  P_{G - w - N(w)}^+(\lambda, \mu).
\end{equation}

For each $k \ge 1$, let $I_k = \{v \in V_k : N(v) \subseteq V_{k-1}\}$, i.e., the set of vertices in $V_k$ that become isolated after we remove $V_{k-1}$ from $G$. For $u\in V_2$, let $e_u$ be the number of its neighbors in $V_1$, and $f_u$ be the number of its neighbors in $V_3$, so that $e_u + f_u = d_u$. 

We have (recall we assume that $G$ is connected, so it has no isolated vertices)
\[
P_{G - w}^+(\lambda, \mu) = (1+\mu)^{|I_1|} \prod_{(u, v)\in E_2} ((1+\mu)^{d_u} + (1+\lambda)^{d_v - 1} - 1)^{\frac{1}{d_u(d_v - 1)}}\prod_{(u, v)\in E_{\ge 3}} ((1+\mu)^{d_u} + (1+\lambda)^{d_v} - 1)^{\frac{1}{d_ud_v}}
\]
and
\[
\lambda P_{G - w - N(w)}^+(\lambda, \mu) = \lambda(1+\lambda)^{|I_2|}\prod_{(u, v)\in E_3} ((1+\mu)^{f_u} + (1+\lambda)^{d_v} - 1)^{\frac{1}{f_ud_v}} \prod_{(u, v) \in E_{\ge 4}} ((1+\mu)^{d_u} + (1+\lambda)^{d_v} - 1)^{\frac{1}{d_ud_v}}.
\]
Thus \eqref{eq:upper-P-recur} expands as
\begin{multline*}
\hspace{-2em} \prod_{(u, v)\in E} ((1+\mu)^{d_u} + (1+\lambda)^{d_v} - 1)^{\frac{1}{d_ud_v}}
\\
\ge (1+\mu)^{|I_1|}\prod_{(u, v)\in E_2} ((1+\mu)^{d_u} + (1+\lambda)^{d_v - 1} - 1)^{\frac{1}{d_u(d_v - 1)}}\prod_{(u, v)\in E_{\ge 3}} ((1+\mu)^{d_u} + (1+\lambda)^{d_v} - 1)^{\frac{1}{d_ud_v}}
\\
\quad +\lambda(1+\lambda)^{|I_2|}\prod_{(u, v)\in E_3} ((1+\mu)^{f_u} + (1+\lambda)^{d_v} - 1)^{\frac{1}{f_ud_v}} \prod_{(u, v) \in E_{\ge 4}} ((1+\mu)^{d_u} + (1+\lambda)^{d_v} - 1)^{\frac{1}{d_ud_v}}.
\end{multline*}
Dividing by $\prod_{(u, v)\in E_{\ge 3}}((1+\mu)^{d_u}+(1+\lambda)^{d_v}-1)^{\frac{1}{d_ud_v}}$, the inequality is equivalent to
\begin{multline}
\prod_{(u, v)\in E_1\cup E_2} ((1+\mu)^{d_u} + (1+\lambda)^{d_v} - 1)^{\frac{1}{d_ud_v}}
\\
\ge (1+\mu)^{|I_1|}\prod_{(u, v)\in E_2} ((1+\mu)^{d_u} + (1+\lambda)^{d_v - 1} - 1)^{\frac{1}{d_u(d_v - 1)}}
\\
\quad + \lambda(1+\lambda)^{|I_2|}\prod_{(u, v)\in E_3} \frac{((1+\mu)^{f_u} + (1+\lambda)^{d_v} - 1)^{\frac{1}{f_ud_v}}}{((1+\mu)^{d_u}+(1+\lambda)^{d_v}-1)^{\frac{1}{d_ud_v}}}. \label{eq:upper-expand-divide}
\end{multline}
By Lemma~\ref{lem:f-increasing-2}, using $f_u \le d_u$ and $d_v \le \Delta$, we have
\[
\frac{((1+\mu)^{f_u} + (1+\lambda)^{d_v} - 1)^{\frac{1}{f_ud_v}}}{((1+\mu)^{d_u}+(1+\lambda)^{d_v}-1)^{\frac{1}{d_ud_v}}} \le 
\frac{((1+\mu)^{f_u} + (1+\lambda)^\Delta  - 1)^{\frac{1}{f_u\Delta}}}{((1+\mu)^{d_u}+(1+\lambda)^\Delta-1)^{\frac{1}{d_u\Delta}}},
\]
so
\begin{align*}
\prod_{(u, v)\in E_3} \frac{((1+\mu)^{f_u} + (1+\lambda)^{d_v} - 1)^{\frac{1}{f_ud_v}}}{((1+\mu)^{d_u}+(1+\lambda)^{d_v}-1)^{\frac{1}{d_ud_v}}}
&\le 
\prod_{(u, v)\in E_3} \frac{((1+\mu)^{f_u} + (1+\lambda)^{\Delta} - 1)^{\frac{1}{f_u\Delta}}}{((1+\mu)^{d_u}+(1+\lambda)^{\Delta}-1)^{\frac{1}{d_u\Delta}}}
\\
&= 
\prod_{u\in V_2 \setminus I_2}\frac{((1+\mu)^{f_u} + (1+\lambda)^{\Delta} - 1)^{\frac{1}{\Delta}}}{((1+\mu)^{d_u} + (1+\lambda)^{\Delta} - 1)^{\frac{f_u}{d_u\Delta}}},
\end{align*}
where in the last step we use that each $u \in V_2$ is contained in exactly $f_u$ edges of $E_3$. Thus, to prove \eqref{eq:upper-expand-divide}, it suffices to show
\begin{multline}
\prod_{(u, v)\in E_1\cup E_2} ((1+\mu)^{d_u} + (1+\lambda)^{d_v} - 1)^{\frac{1}{d_ud_v}}
\\
\ge (1+\mu)^{|I_1|}\prod_{(u, v)\in E_2} ((1+\mu)^{d_u} + (1+\lambda)^{d_v - 1} - 1)^{\frac{1}{d_u(d_v - 1)}} 
\\ 
+ \lambda (1+\lambda)^{|I_2|}\prod_{u\in V_2 \setminus  I_2}\frac{((1+\mu)^{f_u} + (1+\lambda)^{\Delta} - 1)^{\frac{1}{\Delta}}}{((1+\mu)^{d_u} + (1+\lambda)^{\Delta} - 1)^{\frac{f_u}{d_u\Delta}}}. \label{eq:upper-expand-divide-2}
\end{multline}

Apply H\"older's inequality in the form of $a^p + b \le (a+b)^p(1+b)^{1-p}$ for $a,b > 0$ and $p \in [0,1]$ with $a = (1+\mu)^{d_u}$, $b = (1+\lambda)^{\Delta} - 1 $ and $p = f_u/d_u$, we obtain
\begin{equation} \label{eq:upper-holder-app}
(1+\mu)^{f_u} + (1+\lambda)^{\Delta} - 1 \le \left((1+\mu)^{d_u} + (1+\lambda)^{\Delta} - 1\right)^{\frac{f_u}{d_u}} (1+\lambda)^{\Delta(1 - \frac{f_u}{d_u})}.
\end{equation}
Thus
\[
\prod_{u\in V_2 \setminus I_2}\frac{((1+\mu)^{f_u} + (1+\lambda)^{\Delta} - 1)^{\frac{1}{\Delta}}}{((1+\mu)^{d_u} + (1+\lambda)^{\Delta} - 1)^{\frac{f_u}{d_u\Delta}}}
\le
\prod_{u\in V_2 \setminus I_2} (1+\lambda)^{1 - \frac{f_u}{d_u}}.
\]
We have 
$\sum_{u \in V_2 \setminus I_2} (1 - \frac{f_u}{d_u})
= \sum_{u \in V_2 \setminus I_2} \frac{e_u}{d_u}
= \sum_{(u,v) \in E_2} \frac{1}{d_u} -  |I_2|
$
since $e_u$ is the number of edges of $E_2$ containing $u$ as an endpoint. Thus, to prove \eqref{eq:upper-expand-divide-2}, it suffices to show
\begin{multline}
\prod_{(u, v)\in E_1\cup E_2} ((1+\mu)^{d_u} + (1+\lambda)^{d_v} - 1)^{\frac{1}{d_ud_v}} 
\\ \quad \ge (1+\mu)^{|I_1|}\prod_{(u, v)\in E_2} ((1+\mu)^{d_u} + (1+\lambda)^{d_v - 1} - 1)^{\frac{1}{d_u(d_v - 1)}}  + \lambda \prod_{(u, v)\in E_2} (1+\lambda)^{\frac{1}{d_u}} \label{eq:pf-upper-1}.
\end{multline}
Let us upper bound the right-hand side by applying H\"{o}lder's inequality in the form 
\begin{equation}
	\label{eq:upper-holder}
\prod_{i = 1}^k a_i^{p_i} + \prod_{i = 1}^k b_i^{p_i} \le \prod_{i = 1}^k (a_i + b_i)^{p_i}, \quad  \text{where } \sum_{i = 1}^k p_i = 1,
\end{equation}
with the exponents $p_i$ being the summands of 
\[\frac{|I_1|}{\Delta} + \sum_{(u, v)\in E_2}\frac{1}{\Delta (d_v - 1)} =1\]
(as each $v \in V_1$ appears as an endpoint in $d_v -1$ edges of $E_2$). The right-hand-side of \eqref{eq:pf-upper-1} equals
\begin{multline}
	((1+\mu)^\Delta)^{\frac{|I_1|}{\Delta}} 
	\prod_{(u, v)\in E_2} \left(((1+\mu)^{d_u} + (1+\lambda)^{d_v - 1} - 1)^{\frac{\Delta}{d_u}}\right)^{\frac{1}{\Delta(d_v-1)}}  
	\\
	\qquad + \lambda^{\frac{|I_1|}{\Delta}} \prod_{(u, v)\in E_2} \left(\lambda (1+\lambda)^{\frac{\Delta (d_v-1)}{d_u}} \right)^{\frac{1}{\Delta (d_v-1)}}
	\\
	\le
	\left( (1+\mu)^\Delta 
	+ \lambda \right)^{\frac{|I_1|}{\Delta}} \prod_{(u, v)\in E_2} \left(((1+\mu)^{d_u} + (1+\lambda)^{d_v - 1} - 1)^{\frac{\Delta}{d_u}} + \lambda (1+\lambda)^{\frac{\Delta (d_v-1)}{d_u}}\right)^{\frac{1}{\Delta(d_v-1)}} \label{eq:pf-upper-rhs}
\end{multline}
by H\"older's inequality \eqref{eq:upper-holder}. On the other hand, the left-hand side of \eqref{eq:pf-upper-1} may be written as (recall that all edges in $E_1$ have $w$ as an endpoint)
\begin{multline}
\prod_{v \in V_1} ((1+\mu)^\Delta  + (1+\lambda)^{d_v} - 1)^{\frac{1}{\Delta d_v}}
\prod_{(u, v)\in E_2} ((1+\mu)^{d_u} + (1+\lambda)^{d_v} - 1)^{\frac{1}{d_ud_v}}
\\
\quad = ((1+\mu)^{\Delta} + \lambda)^{\frac{|I_1|}{\Delta}}\prod_{(u, v)\in E_2} ((1+\mu)^{\Delta} + (1+\lambda)^{d_v} - 1)^{\frac{1}{\Delta d_v(d_v - 1)}}((1+\mu)^{d_u} + (1+\lambda)^{d_v} - 1)^{\frac{1}{d_ud_v}}, \label{eq:pf-upper-lhs}
\end{multline}
obtained by distributing each $((1+\mu)^\Delta  + (1+\lambda)^{d_v} - 1)^{\frac{1}{\Delta d_v}}$ factor on the left-hand side evenly over all edges of $E_2$ containing $v$, noting that the exponents add up as $\frac{1}{\Delta d_v} = \sum_{u : (u,v)\in E_2} \frac{1}{\Delta d_v(d_v-1)}$ for each $v \in V_1$. It remains to show that the right-hand side of \eqref{eq:pf-upper-rhs} is at most \eqref{eq:pf-upper-lhs}, which would follow if for every $(u,v) \in E_2$,
\begin{multline*}
((1+\mu)^{d_u} + (1+\lambda)^{d_v - 1} - 1)^{\frac{\Delta}{d_u}} + \lambda (1+\lambda)^{\frac{\Delta (d_v-1)}{d_u}}
\\
\le  ((1+\mu)^{\Delta} + (1+\lambda)^{d_v} - 1)^{\frac{1}{d_v}}((1+\mu)^{d_u} + (1+\lambda)^{d_v} - 1)^{\frac{\Delta(d_v-1)}{d_ud_v}}.
\end{multline*}
By Lemma~\ref{lem:u-v-w-c-2}, this inequality holds for all reals $1 \le d_u, d_v \le \Delta$ and $\lambda,\mu > 0$.

\smallskip

\emph{Equality conditions.} Suppose equality occurs in Theorem~\ref{thm:upper-biweighted}. We still assume that $G$ is connected. Since $\lambda, \mu > 0$, Lemma~\ref{lem:u-v-w-c-2} further implies that $d_u=\Delta$ or $d_v=1$ for all $(u,v)\in E_2$. Notice that every $v$ with $(u, v)\in E_2$ has $d_v \ge 2$, so $d_u = \Delta$ for all $u \in V_2$. To have equality in \eqref{eq:upper-holder-app}, we must have $f_u\in\{0, d_u\}$ for every $u\in V_2\setminus I_2$, since to attain equality in H\"older's inequality $a^p + b\le (a + b)^p(1 + b)^{1 - p}$ with $a > 1$ and $b > 0$, we must have $p\in\{0, 1\}$. But $e_u = d_u - f_u\ge 1$ by definition of $V_2$, and thus $f_u = 0$ for all $u\in V_2\setminus I_2$, and hence $V_2 = I_2$, which implies that $G = K_{\Delta, d_v}$ for some $v \in V_1$.

The inequality is strict for all connected $G$ except for complete bipartite graphs. Since the inequality factors over connected components, we see that equality occurs for a general graph $G$ if and only if  $G$ is a disjoint union of complete bipartite graphs. This completes the proof of Theorem~\ref{thm:upper-biweighted}.

\section{Further remarks} \label{sec:concluding}

\subsection{Degree conditions}

As a corollary of our main theorems, we obtain tight bounds on the exponentially normalized number $i(G)^{1/|V(G)|}$ of independent sets of a graph $G$ subject to the degree distribution of $G$, i.e., the fraction of vertices of every degree. (The minimization problem is actually equivalent to Theorem~\ref{thm:lower-weighted}.)

Let $\lambda > 0$ and let $\bm \rho = (\rho_0, \rho_1, \dots)$ be a finitely supported sequence of nonnegative rational numbers summing to 1. Let $f_{\min}(\bm \rho; \lambda)$ and $f_{\max}(\bm \rho; \lambda)$ denote the minimum and maximum possible values, respectively, of $P_G(\lambda)^{1/|V(G)|}$, over all graphs $G$ with degree distribution $\rho$, i.e., exactly $\rho_i |V(G)|$ vertices of $G$ have degree $i$ for each $i \ge 0$.

Theorem~\ref{thm:lower-weighted} says us that the minimum possible value of $P_G(\lambda)^{1/|V(G)|}$ is attained by a disjoint union of cliques, so that
\[
f_{\min}(\bm \rho;\lambda) = \prod_{i \ge 0} P_{K_{i+1}}(\lambda)^{\frac{\rho_i}{i+1}}.
\]

Theorem~\ref{thm:upper-weighted} implies that the maximum possible value of $P_G(\lambda)^{1/|V(G)|}$ is attained by a disjoint union of complete bipartite graphs, where the vertices of largest degree are paired with the vertices of smallest degree successively in a greedy fashion, assuming that the number of vertices satisfies appropriate divisibility conditions. We give the corresponding function $f_{\max}(\bm\rho; \lambda)$ recursively. We expand the domain of $f_{\max}$ by dropping the requirement that $\bm\rho$ sums to 1. Let $\Delta(\bm \rho)$ and $\delta(\bm \rho)$ denote the largest and smallest nonzero indices in $\bm \rho$, respectively, with $\Delta(\bm\rho)=\delta(\bm \rho)=-1$ if these indices do not exist. Finally, let $\bm e_i$ denote the sequence $\bm\rho=(\rho_0, \rho_1,\dots)$ with $\rho_i = 1$ and $\rho_j=0$ for all $j \ne i$. We claim that $f_{\max}$ is given by the recursion: writing $\delta = \delta(\bm\rho)$ and $\Delta = \Delta(\bm\rho)$,
\[f_{\max}(\bm \rho; \lambda)= 
\begin{cases}
\qquad \qquad \qquad\qquad \qquad \qquad 2^{\rho_{0}}  & \text{if } \delta=-1,
\\ P_{K_{\delta, \Delta}}(\lambda)^{\rho_{\delta}/\Delta}f_{\max}(\bm \rho-\rho_{\delta}\bm e_{\delta}-\frac{\delta}{\Delta}\rho_{\delta}\bm e_{\Delta} ; \lambda) & \text{if } \delta \neq -1 \text{ and }\delta\rho_{\delta}\le \Delta \rho_{\Delta},
\\ P_{K_{\delta, \Delta}}(\lambda)^{\rho_{\Delta}/\delta}f_{\max}(\bm \rho-\frac{\Delta}{\delta}\rho_{\Delta} \bm e_{\delta}-\rho_{\Delta}\bm e_{\Delta} ; \lambda) & \text{if } \delta\neq -1 \text{ and } \Delta \rho_{\Delta}\le \delta \rho_{\delta}.
\end{cases}
\]
This recursion terminates after a finite number of steps, since the support of $\bm\rho$ becomes strictly smaller at each step.

The claim follows from Theorem 1.5 along with the following observation. If $a < b$ and $c < d$, and $G$ contains $bd$ copies of $K_{a, c}$ and $ac$ copies of $K_{b, d}$, then by replacing them by $bc$ copies of $K_{a,d}$ and $ad$ copies of $K_{b,c}$, we never decrease $P_G(\lambda)$, as $P_{K_{a, c}}(\lambda)^{bd}P_{K_{b, d}}(\lambda)^{ac} \le P_{K_{a, d}}(\lambda)^{bc}P_{K_{b, c}}(\lambda)^{ad}$ by Lemma~\ref{lem:f-increasing-2}. Note that this operation does not change the degree distribution of the graph.

Given any $G$ that is a disjoint union of complete bipartite graphs, after taking an appropriate number of disjoint copies of $G$, we may successively apply the above operation so that, at the end of the process, we have a disjoint union of complete bipartite graphs where the edges consist of the largest degree vertices successively paired off with the smallest degree vertices. It is easy to see that there is a unique such pairing as long as the number of vertices is highly divisible (which is true as we took many disjoint copies of the graph in an earlier step), and the maximum value of $P_G(\lambda)^{1/|V(G)|}$ corresponds to the $f_{\max}$ stated above.

\medskip

A similar procedure lets us obtain the extrema for $P_G(\lambda)^{1/|V(G)|}$ subject to conditions on the minimum/average/maximum degree of $G$. The expressions are somewhat complicated, so we do not include them here.

\subsection{Bounds on independence number} We note a couple of neat corollaries. Theorem~\ref{thm:lower-weighted} says that
\[
\sum_{I\in\mathcal{I}(G)} \lambda^{|I|} \ge \prod_{v \in V(G)} ((d_v+1)\lambda + 1)^{1/(d_v+1)}.
\]
Letting $\lambda \to \infty$ and comparing the growth rate of the two sides, we obtain the following lower bound on the independence number $\alpha(G)$ (the size of the largest independent set of $G$):
\[
\alpha(G) \ge \sum_{v\in V(G)}\frac{1}{d_v+1}.
\]
This is actually the classic Caro--Wei bound \cite{Caro79,Wei81}, from which Tur\'an's theorem can be deduced by noting that the right-hand side is, by convexity, at least $|V(G)|/(\overline d + 1)$, where $\overline d$ is the average degree in $G$. The Caro--Wei bound has a short probabilistic proof (taken from~\cite{AS}): randomly order the vertices of $G$ and consider the independent set where we include a vertex if it appears before all its neighbors. The right-hand side above is the expected size of this independent set.

Similarly, starting with Theorem~\ref{thm:upper-weighted}, which says
\[
\sum_{I\in\mathcal{I}(G)} \lambda^{|I|} 
\le (1+\lambda)^{\iso(G)} \prod_{uv \in E(G)} ((1+\lambda)^{d_u} + (1+\lambda)^{d_v}-1)^{1/(d_ud_v)},
\]
and taking $\lambda \to \infty$, we have
\[
\alpha(G) \le \sum_{(u, v)\in E(G)}\frac{1}{\min (d_u, d_v)}+\iso(G).
\]
This inequality also has a quick proof: given an independent set $I$, for each $v \in I$, assign weight $1/d_v$ to all edges incident to $v$, and note that the right-hand side upper bounds the sum of the edge-weights.

\subsection{Extensions to colorings and graph homomorphisms}

Let $c_q(G)$ denote the number of $q$-colorings of a graph $G$. The following conjecture of Galvin and Tetali~\cite{GT04} remains one of the most interesting open problems on this topic.

\begin{conjecture} \label{conj:color}
	For $q\ge 3$ and $n$-vertex $d$-regular graph $G$,
	\[
		c_q(G) \le c_q(K_{d,d})^{n/(2d)}.
	\]
\end{conjecture}

Galvin and Tetali proved the result for bipartite $G$ (analogous to Kahn's \cite{Kahn01} bound on independent sets). Zhao's bipartite swapping trick~\cite{Zhao10,Zhao11} did not extend to $q$-colorings. Very recently, the $d=3$ case was proved by Davies, Jenssen, Perkins, and Roberts~\cite{DJPR3} using the occupancy method (along with a computer-aided verification), and it was later extended to $d=4$ \cite{Dav}.

To tackle this conjecture using our methods, one needs to formulate a more general conjecture, e.g., \cite{Gal06}
\[
c_q(G) \le \prod_{uv \in E(G)} c_q(K_{d_u, d_v})^{1/(d_ud_v)}.
\]
However, the number of colorings does not have the nice recursive relation $i(G) = i(G-w) + i(G-w-N(w))$ for independent sets. A natural workaround is to consider list-colorings, i.e., assign every vertex $v$ a list $L_v$ of possible remaining colors. Then there is an easy recursive relation on the number of list colorings: for each possible color assignment to $w$, delete $w$ from $G$, and remove the assigned color from the lists of the neighbors of $w$. 

More generally, Galvin and Tetali~\cite{GT04} proved that the number, $\hom(G, H)$,  of graph homomorphisms  from $G$ to $H$, where $H$ is a fixed graph allowing loops, satisfies the following inequality: for every $n$-vertex $d$-regular \emph{bipartite} graph $G$,
\[
\hom(G, H) \le \hom(K_{d,d}, H)^{n/(2d)}.
\]
This general setup includes independent sets as $\hom(G, \tikzHind) = i(G)$. It also includes $q$-colorings as $\hom(G, K_q) = c_q(G)$. The bipartite assumption on $G$ cannot be relaxed in general, for example by taking $H$ to be two looped vertices. Nonetheless, there are lots of interesting results and conjectures regarding what happens when one relaxes the bipartiteness assumption. See the survey \cite{Zhao17}.

It was conjectured \cite{CCPT} that $\hom(G, H) \le \hom(K_{d,d}, H)^{n/(2d)}$ for all \emph{triangle-free} $G$. Furthermore, as with Theorem~\ref{thm:upper}, it was conjectured~\cite{Gal06}\footnote{The triangle-free assumption was missing in \cite{Gal06}.} that for all triangle-free $G$,
\[
\hom(G, H) \le \prod_{uv \in E(G)} \hom(K_{d_u,d_v}, H)^{1/(d_ud_v)}.
\]

We believe that these conjectures are amenable to our methods. We plan to address them in a follow-up work.

\textbf{Notes added.} We proved all conjectures mentioned above in our follow-up work \cite{SSSZ2}. The methods in \cite{SSSZ2} would also give a more streamlined proof of Theorem~\ref{thm:upper-biweighted}, eliminating the need for the calculus verifications in the appendix.

\section*{Acknowledgments} We thank David Galvin and P\'eter Csikv\'ari for helpful comments on the paper.

\appendix 

\section{Some analytic inequalities} \label{sec:ineq}

This appendix contains a number of technical inequalities used in the proof of the main theorems. 

\begin{lemma}\label{lem:g-derivative}
Fix $\Delta\ge 1$ and $\lambda >0$. Then the function 
\[g(x)=(\Delta\lambda+1)^\frac{\Delta-1}{\Delta}(x\lambda+1)^\frac{1}{x}+\lambda- ((\Delta+1)\lambda+1)^\frac{\Delta}{\Delta+1}((x+1)\lambda+1)^\frac{1}{x+1}\] 
is strictly decreasing for $0 < x < \Delta$. In particular, $g(x) > 0$ for $0 < x < \Delta$ because $g(\Delta) = 0$.
\end{lemma}

\begin{proof}
We need $g'(x) < 0$. We have
\begin{align*} 
g'(x)&=\frac{(\Delta\lambda+1)^\frac{\Delta-1}{\Delta}}{(x\lambda+1)^\frac{x-1}{x}}\Bigg(\frac{\lambda}{x}-\frac{(x\lambda+1)(\log(x\lambda+1))}{x^2}\Bigg)\\
&\qquad -\frac{((\Delta+1)\lambda+1)^\frac{\Delta}{\Delta+1}}{((x+1)\lambda+1)^\frac{x}{x+1}}\Bigg(\frac{\lambda}{x+1}-\frac{((x+1)\lambda+1)\log((x+1)\lambda+1)}{(x+1)^2}\Bigg).
\end{align*}
We have $\frac{(x\lambda+1)\log(x\lambda+1)}{x^2}-\frac{\lambda}{x} > 0$ and $\frac{((x+1)\lambda+1)\log((x+1)\lambda+1)}{(x+1)^2}-\frac{\lambda}{x+1} > 0$ since they both follow from the inequality $(y + 1)\log (y + 1) - y > 0$ for $y > 0$, equivalent to $\log\left(\frac{1}{y + 1}\right) < \frac{1}{y + 1} - 1$. Thus it suffices to prove that
\begin{equation}\label{eq:a11}
\frac{(x\lambda+1)\log(x\lambda+1)}{x^2}-\frac{\lambda}{x}\ge \frac{((x+1)\lambda+1)\log((x+1)\lambda+1)}{(x+1)^2}-\frac{\lambda}{x+1}	
\end{equation}
and 
\begin{equation}\label{eq:a12}
\frac{(x\lambda+1)^\frac{x-1}{x}}{((x+1)\lambda+1)^\frac{x}{x+1}} < \frac{(\Delta\lambda+1)^\frac{\Delta-1}{\Delta}}{((\Delta+1)\lambda+1)^\frac{\Delta}{\Delta+1}}.
\end{equation}
The inequality \eqref{eq:a11} follows as the function 
\[h(x)=\frac{(x\lambda+1)\log(x\lambda+1)}{x^2}-\frac{\lambda}{x}\] 
is nonincreasing, as
\[h'(x)=\frac{2\lambda x-(\lambda x+2)\log(\lambda x+1)}{x^3}\le 0,\]
where we used $(y + 2)\log (y + 1) - 2y\ge 0$ for $y\ge 0$, which is true since it is true at $y = 0$ and its derivative is $\frac{y + 2}{y + 1} + \log (y + 1) - 2 = \frac{1}{y + 1} - 1 - \log\left(\frac{1}{y + 1}\right)\ge 0$.

The inequality \eqref{eq:a12} follows by proving that the function 
\[k(x)=(x\lambda +1)^\frac{x-1}{x}\]
is strictly log-concave on $(0, \infty)$. This reduces to showing that
\[\frac{d^2}{dx^2}[\log k(x)]=\frac{x\lambda(2 + 3x\lambda - x^2\lambda)-2(x\lambda+1)^2\log(x\lambda+1)}{x^3(x\lambda+1)^2} > 0,\]
which is equivalent to
\[
	\frac{x\lambda(2 + 3x\lambda - x^2\lambda)}{2(x\lambda+1)^2} < \log (x\lambda + 1)
\] 
for $x > 0$. This is true since there is equality at $x = 0$ and
\[\frac{d}{dx}\left[\log (x\lambda + 1)-\frac{x\lambda(2 + 3x\lambda - x^2\lambda)}{2(x\lambda+1)^2}\right] = \frac{x^2\lambda^2(\lambda (x + 2)+3)}{2(x\lambda+1)^3} > 0. \qedhere\]
\end{proof}

\begin{lemma}\label{lem:f-increasing-2}
For $\beta\ge\alpha > 0$ and $\lambda,\mu \ge 0$. Then the function
\[f(x)=\frac{((1+\mu)^\alpha+(1+\lambda)^x-1)^{\frac{1}{\alpha x}}}{((1+\mu)^{\beta}+(1+\lambda)^x-1)^{\frac{1}{\beta x}}}\]
is nondecreasing on $(0, \infty)$. 
\end{lemma}

\begin{proof}
Let
\[
	r(x, y) := x\log x+y\log y+\frac{xy}{x+y-1}\log x\log y - (x+y-1)\log (x+y-1).
\]
Let us show that, for $x,y \ge 1$, we have $r(x, y)\ge 0$, which is equivalent to
\[q(x)q(y)\ge q(1)q(x+y-1)\]
where $q(t)=1+\frac{t}{x+y-1}\log t$. Consider the function 
\[
p(\alpha)=\log q(c-\alpha)+\log q(c+\alpha),
\]
where $c=\frac{x+y}{2}$. It suffices to check that $p'(\alpha)$ is nonpositive when $\alpha\in [0, c-1]$, since, upon taking logs, the left side is at $\alpha = \frac{|x - y|}{2}$ and the right side is at $\alpha = \frac{x + y}{2} - 1$, which is larger since $x, y\ge 1$.

Taking derivative, we have
\[p'(\alpha)=\frac{(c-\alpha-1)\log(c+\alpha)-\log(c-\alpha)(c+\alpha-1)-2\alpha\log (c+\alpha)\log (c-\alpha)}{((c-\alpha)\log(c-\alpha)+2c-1)((c+\alpha)\log(c+\alpha)+2c-1)}.\]
Since $c\ge 1$ and $\alpha\in [0, c-1]$, it follows that the denominator in this expression is positive. Furthermore, for $\alpha>0$ the numerator is at most
\[\gamma(\alpha)=(c-\alpha-1)\log (c+\alpha)-(c+\alpha-1)\log (c-\alpha).\]
Its second derivative is 
\[
\gamma''(\alpha)=\frac{2\alpha(5c^2-\alpha^2-2c)}{(c+\alpha)^2(c-\alpha)^2}>0,
\]
since $c^2>\alpha^2$ and $c^2\ge c$. Hence, in order to verify $\gamma(\alpha)\le 0$ for $0\le \alpha\le c-1$, it suffices to check the endpoints. In fact $\gamma(0)=\gamma(c-1)=0$, so $\gamma$ is indeed nonpositive. Hence $r(x, y)\ge 0$ for $x, y\ge 1$ as required.

Now we return to the inequality stated in the lemma. It suffices to check that $\log f(x)$ is nondecreasing on this interval, since clearly $f$ takes positive values. We have
\[\frac{\partial}{\partial x}\log f(x)=s(x, \alpha)-s(x, \beta)\]
where
\[s(x, t)=\frac{(1+\lambda)^x\log (1+\lambda)}{tx((1+\mu)^t+(1+\lambda)^x-1)}-\frac{\log((1+\mu)^t+(1+\lambda)^x-1)}{tx^2}.\]
Since $\alpha\le\beta$, it suffices to check that the partial derivative $\partial s /\partial t$, which is true since
\[\frac{\partial s(x, t)}{\partial t} =-\frac{r((1+\lambda)^x, (1+\mu)^t)}{x^2t^2((1+\lambda
_1)^x+(1+\mu)^t-1)} \le 0\]
by our inequality $r(x,y) \ge 0$ for all $x,y\ge 1$.
\end{proof}

\begin{lemma}\label{lem:u-v-w-c-2}
Let $c_1, c_2\ge 0$, and let $u, v, w\ge 1$ be positive reals with $1\le u\le w$ and $1\le v\le w$. Then
\begin{multline*}
(1+c_1)^{(v-1)\frac{w}{u}}c_1+\left[(1+c_1)^{v-1}+(1+c_2)^{u}-1\right]^{\frac{w}{u}}
\\ \le \left[(1+c_1)^v+(1+c_2)^u-1\right]^{\frac{w(v-1)}{uv}}((1+c_1)^v+(1+c_2)^w-1)^{\frac{1}{v}}.
\end{multline*}
Equality holds if and only if $v = 1$ or $w = u$ or $c_1c_2 = 0$.
\end{lemma}

\begin{proof}
When $v=1$, equality holds since both sides evaluate to $(1+c_2)^w+c_1$. Similarly, if $c_1=0$, then both sides evaluate to $(1+c_2)^w$ and if $c_2=0$ then both sides evaluate to $(1+c_1)^{\frac{w(v-1)}{u}+1}$. Hence, we will assume $v>1$ and $c_1,c_2>0$ from now on.

Applying H\"{o}lder's inequality to the left-hand side, we get
\begin{align*}
    &(1+c_1)^{(v-1)\frac{w}{u}}c+((1+c_1)^{v-1}+(1+c_2)^{u}-1)^{\frac{w}{u}}
    \\&= \left[c_1(1+c_1)^{v-1}\right]^{\frac{1}{v}}\left[c_1(1+c_1)^{\frac{wv}{u}-1}\right]^{\frac{v-1}{v}}
    \\ &\qquad +\left[(1+c_2)^w+(1+c_1)^{v-1}-1\right]^{\frac{1}{v}}\left[\left[\frac{\left[(1+c_1)^{v-1}+(1+c_2)^u-1\right]^{\frac{vw}{u}}}{(1+c_1)^{v-1}+(1+c_2)^w-1}\right]^{\frac{1}{v-1}}\right]^{\frac{v-1}{v}}
    \\ &\le \left[(1+c_2)^w+(1+c_1)^{v}-1\right]^{\frac{1}{v}}\left[c_1(1+c_1)^{\frac{wv}{u}-1}+\left[\frac{\left[(1+c_1)^{v-1}+(1+c_2)^u-1\right]^{\frac{vw}{u}}}{(1+c_1)^{v-1}+(1+c_2)^w-1}\right]^{\frac{1}{v-1}}\right]^{\frac{v-1}{v}}.
\end{align*}
Therefore it suffices to prove that
\[c_1(1+c_1)^{\frac{wv}{u}-1}+\left[\frac{\left[(1+c_1)^{v-1}+(1+c_2)^u-1\right]^{\frac{vw}{u}}}{(1+c_1)^{v-1}+(1+c_2)^w-1}\right]^{\frac{1}{v-1}}\le \left[(1+c_1)^v+(1+c_2)^u-1\right]^{\frac{w}{u}}.\]

Let $a=(1+c_1)^{v-1}, b=(1+c_2)^u-1, t=\frac{w}{u}, c=c_1$. Then upon dividing through by the right hand side, the above inequality can be rewritten as $F_1+F_2\le 1$, where
\[
F_1=\Bigg(\frac{ac}{a+ac+b}\Bigg)^t\Bigg(\frac{1+c}{c}\Bigg)^{t-1} \quad\text{and}\quad 
    F_2= \Bigg(\frac{a+b}{a+ac+b}\Bigg)^t\Bigg(\frac{(a+b)^t}{a+(b+1)^t-1}\Bigg)^{\frac{\log (1+c)}{\log a}}.\]
It suffices to prove that for all $a> 1, b > 0, c>0, t\ge 1$, one has $F_1+F_2\le 1$. Fix $a > 1$, $b > 0$, and $c > 0$. Set
\[
F(t):=\log F_2-\log (1-F_1).
\]
We need to show that $F(t) \le 0$ for all $t \ge 1$ with equality if and only if $t = 1$. We have $F(1) = 0$, so it suffices to check that $F'(t) < 0$ for $t>1$, which follows from the following two facts
\begin{enumerate}
    \item[(A)] $F'(1) <  0$. 
    \item[(B)] There exists a function $M(t)$ which is positive on $(1, \infty)$, and for which $M(t)F'(t)$ is nonincreasing.
\end{enumerate}

We have
\[
F'(t) =\log\left(\frac{a+b}{a+ac+b}\right)+\frac{\log(1+c)\log(a+b)}{\log a}-\frac{\log(1+c)\log(1+b)}{(\log a) \Big(1+\frac{a-1}{(b+1)^t}\Big)}-\frac{\log\Big(\frac{a+ac}{a+ac+b}\Big)}{1-\Big(\frac{a+ac}{a+ac+b}\Big)^{-t}\Big(1+\frac{1}{c}\Big)}.\]

\emph{Proof of (A).}  We have
\[
F'(1)=\log\Big(\frac{a+b}{a+ac+b}\Big)+\frac{\log(1+c)\log(a+b)}{\log a}-\frac{\log(1+c)\log(1+b)(1+b)}{\log a (a+b)}+\frac{\log\Big(\frac{a+ac}{a+ac+b}\Big)ac}{a+b}.\]
Hence $F'(1) < 0$ is equivalent to, upon multiplying through by $(a+b)\log a$, substituting $d=c + 1$, and rearranging,
\[(a+b)\log(ad)\log(a+b)+(ad-a)\log(ad)\log a < (ad+b)\log(ad+b)\log a+(1+b)\log(1+b)\log d.\]
Note that both sides are equal if $b=0$. We claim that the difference ($\mathrm{RHS}-\mathrm{LHS}$) is strictly increasing in $b$. Indeed, upon taking a derivative this is equivalent to
\[\log(ad)\log(a+b) < \log a\log(ad+b)+\log d\log(1+b),\]
which is in turn equivalent to, upon dividing through by $\log(ad)>0$,
\[\log(a+b)\le\frac{\log a}{\log a+\log d}\log(ad+b)+\frac{\log d}{\log a+\log d}(1+b),\]
which follows from Jensen's inequality on the strictly convex function $r(x)=\log(e^x+b)$ and the fact that $ad > 1, b > 0$. This completes the proof of (A).

\medskip

\emph{Proof of (B).} Set $d=c+1$, so that

\[
F'(t) = \log\left(\frac{a+b}{ad+b}\right)+\frac{\log d\log(a+b)}{\log a}-\frac{\log d\log(1+b)}{\log a\cdot \Big(1+\frac{a-1}{(b+1)^t}\Big)}-\frac{\log(\frac{ad+b}{ad})}{\Big(\frac{ad+b}{ad}\Big)^{t}\Big(\frac{d}{d-1}\Big)-1}.\]
Set
\[
M(t) =\Big(\frac{ad+b}{ad}\Big)^t-\frac{d-1}{d}.
\]
We have $M(t) > 0$ for $t > 1$ since $d>1$ and $b>0$. Note that $M'(t)=\log (\frac{ad+b}{ad})(\frac{ad+b}{ad})^t$. We compute:
\begin{align*}
\frac{d}{dt}(M(t)F'(t))&=M'(t) \Bigg(\frac{\log d\log(b+a)}{\log a}+\log\Big(\frac{a+b}{ad+b}\Big)\Bigg)
\\ &\qquad -M'(t)\frac{\log d\log(1+b)}{\log a\cdot \Big(1+\frac{a-1}{(b+1)^t}\Big)}-M(t)\frac{\log d\log^2 (b+1)(a-1)(b+1)^t}{\log a[(b+1)^t+a-1]^2}
\\ &=\log\Big(\frac{ad+b}{ad}\Big)\Big(\frac{ad+b}{ad}\Big)^t\Bigg(\frac{\log d\log(b+a)}{\log a}+\log\Big(\frac{a+b}{ad+b}\Big)\Bigg)
\\ &\qquad - \frac{\log d\log(1+b)}{\log a}\left[\Big(\frac{ad+b}{ad}\Big)^t\log\Big(\frac{ad+b}{ad}\Big)\Bigg(\frac{(b+1)^t}{(b+1)^t+a-1}\Bigg)\right.
\\ &\qquad + \left.\log(b+1)(a-1)\Bigg(\Big(\frac{ad+b}{ad}\Big)^t-\frac{d-1}{d}\Bigg)\frac{(b+1)^t}{\left[(b+1)^t+a-1\right]^2}\right].
\end{align*}

We wish to show this is nonpositive. For this, we first invoke the estimate
\[\frac{\log d\log(b+a)}{\log a}+\log\Big(\frac{a+b}{ad+b}\Big)\le \frac{\log d\log(b+1)}{\log a}.\]
Indeed, this is equivalent to
\[\frac{\log a}{\log a+\log d}\log(ad+b)+\frac{\log d}{\log a+\log d}\log (1+b) \ge \log(a+b),\]
which follows from Jensen's inequality applied to the convex function $x \mapsto \log(e^x+b)$. Using this estimate and dividing through by $\Big(\frac{ad}{ad+b}\Big)^t\frac{\log d\log(b+1)}{\log a}>0$, it suffices to check that
\[
	\log\Bigg(\frac{ad+b}{ad}\Bigg)\Bigg(\frac{a-1}{a-1+(b+1)^t}\Bigg) \le \log(b+1)(a-1)\Bigg(1-\frac{d-1}{d}\Big(\frac{ad}{ad+b}\Big)^t\Bigg)\frac{(b+1)^t}{[(b+1)^t+a-1]^2}.
\]
Multiplying through by $\frac{[(b+1)^t+a-1]^2}{(b+1)^t(a-1)}>0$, this is equivalent to
\[
	\log\Bigg(\frac{ad+b}{ad}\Bigg)\Bigg(1+\frac{a-1}{(b+1)^t}\Bigg)\le \log(b+1)\Bigg(1-\frac{d-1}{d}\Big(\frac{ad}{ad+b}\Big)^t\Bigg).
\]
The left side is decreasing in $t$, while the right side is increasing in $t$. Hence it suffices to check the inequality at $t=1$, which simplifies to
\[\log\Big(\frac{ad+b}{ad}\Big)(ad+b)\le \log(b+1)(b+1).\]
This follows from $ad>1$, along with the fact that the function $x \mapsto (x+b)\log(\frac{x+b}{x})$ is decreasing as the derivative is $\log(1+\frac{b}{x})-\frac{b}{x}\le 0$. This completes the proof of (B), and hence the proof of the lemma.

Tracing out the equality conditions, we saw that in the case that $c_1, c_2 > 0$ and $v > 1$, equality holds exactly when $t = 1$, that is, $w = u$.
\end{proof}


\begin{thebibliography}{10}

\bibitem{AS}
N.~Alon and J.~H.~Spencer, \emph{The probabilistic method}, fourth ed., 
  Wiley Series in Discrete Mathematics and Optimization, John Wiley \& Sons, Inc., Hoboken, NJ, 2016.

\bibitem{Alon91}
N.~Alon, \emph{Independent sets in regular graphs and sum-free subsets of
  finite groups}, Israel J. Math. \textbf{73} (1991), 247--256.

\bibitem{BMS15}
J.~Balogh, R.~Morris, and W.~Samotij, \emph{Independent sets in hypergraphs},
  J. Amer. Math. Soc. \textbf{28} (2015), 669--709.

\bibitem{BMS18}
J.~Balogh, R.~Morris, and W.~Samotij, \emph{The method of hypergraph
  containers}, Proceedings of the International Congress of
  Mathematicians---Rio de Janeiro 2018, to appear.
  
\bibitem{Caro79}
Y.~Caro, \emph{New results on the independence number}, Tech.
  report, Tel-Aviv University, 1979.
  
\bibitem{CFGS86}
F.~R.~K.~Chung, P.~Frankl, R.~Graham and J.B.~Shearer, 
  \emph{Some intersection theorems for ordered sets and graphs}, 
  J. Combin. Theory Ser. A. \textbf{48} (1986), 23--37.

\bibitem{CCPT}
E.~Cohen, P.~Csikv{\'a}ri, W.~Perkins, and P.~Tetali, \emph{The
  {W}idom-{R}owlinson model, the hard-core model and the extremality of the
  complete graph}, European J. Combin. \textbf{62} (2017), 70--76.

\bibitem{CR14}
J.~Cutler and A.~J. Radcliffe, \emph{The maximum number of complete subgraphs
  in a graph with given maximum degree}, J. Combin. Theory Ser. B \textbf{104}
  (2014), 60--71.

\bibitem{Dav}
E.~Davies, \emph{Counting proper colourings in 4-regular graphs via the {P}otts
  model}, arXiv:1801.07547.

\bibitem{DJPR1}
E.~Davies, M.~Jenssen, W.~Perkins, and B.~Roberts, \emph{Independent sets,
  matchings, and occupancy fractions}, J. Lond. Math. Soc. (2) \textbf{96}
  (2017), 47--66.

\bibitem{DJPR2}
E.~Davies, M.~Jenssen, W.~Perkins, and B.~Roberts, \emph{On the average size of
  independent sets in triangle-free graphs}, Proc. Amer. Math. Soc.
  \textbf{146} (2018), 111--124.

\bibitem{DJPR3}
E.~Davies, M.~Jenssen, W.~Perkins, and B.~Roberts, \emph{Extremes of the
  internal energy of the {P}otts model on cubic graphs}, Random Structures
  Algorithms, to appear.
    
\bibitem{Fri04}
E.~Friedgut,
  \emph{Hypergraphs, entropy, and inequalities},
  Amer. Math. Monthly \textbf{111} (2004), 749--760.

\bibitem{GT04}
D.~Galvin and P.~Tetali, \emph{On weighted graph homomorphisms}, Graphs,
  morphisms and statistical physics, DIMACS Ser. Discrete Math. Theoret.
  Comput. Sci., vol.~63, Amer. Math. Soc., Providence, RI, 2004, pp.~97--104.

\bibitem{GZ11}
D.~Galvin and Y.~Zhao, \emph{The number of independent sets in a graph with
  small maximum degree}, Graphs Combin. \textbf{27} (2011), 177--186.

\bibitem{Gal06}
D.~J.~Galvin, \emph{Bounding the partition function of spin-systems}, Electron.
  J. Combin. \textbf{13} (2006), Research Paper 72, 11.
  
\bibitem{Gal14}
D.~Galvin, \emph{Three tutorial lectures on entropy and counting}, arXiv:1406.7872.

\bibitem{JJP1}
M.~Jenssen, F.~Joos, and W.~Perkins, \emph{On the hard sphere model and sphere
  packings in high dimensions}, arXiv:1707.00476.

\bibitem{JJP2}
M.~Jenssen, F.~Joos, and W.~Perkins, \emph{On kissing numbers and spherical
  codes in high dimensions}, arXiv:1803.02702.

\bibitem{Kahn01}
J.~Kahn, \emph{An entropy approach to the hard-core model on bipartite graphs},
  Combin. Probab. Comput. \textbf{10} (2001), 219--237.

\bibitem{LZ14}
E.~Lubetzky and Y.~Zhao, 
  \emph{On replica symmetry of large deviations in random graphs}, 
  Random Structures Algorithms \textbf{47} (2014) 109--146.
  
\bibitem{MT10}
M.~Madiman and P.~Tetali, 
  \emph{Information inequalities for joint distributions, with interpretations and applications}, 
  IEEE Trans. on Information Theory \textbf{56} (2010), 2699--2713.

\bibitem{SSSZ2}
A. Sah, M. Sawhney, D. Stoner, and Y. Zhao,
\emph{A reverse Sidorenko inequality}, arXiv:1809.09462.

\bibitem{Sap01}
A.~A. Sapozhenko, \emph{On the number of independent sets in extenders},
  Diskret. Mat. \textbf{13} (2001), 56--62.
  

\bibitem{ST15}
D.~Saxton and A.~Thomason, \emph{Hypergraph containers}, Invent. Math.
  \textbf{201} (2015), 925--992.

\bibitem{Wei81}
V.~K. Wei, \emph{A lower bound on the stability number of a simple graph},
  Tech. report, Bell Lab., 1981.

\bibitem{Zhao10}
Y.~Zhao, \emph{The number of independent sets in a regular graph}, Combin.
  Probab. Comput. \textbf{19} (2010), 315--320.

\bibitem{Zhao11}
Y.~Zhao, \emph{The bipartite swapping trick on graph homomorphisms}, SIAM J.
  Discrete Math. \textbf{25} (2011), 660--680.

\bibitem{Zhao17}
Y.~Zhao, \emph{Extremal regular graphs: independent sets and graph
  homomorphisms}, Amer. Math. Monthly \textbf{124} (2017), 827--843.

\end{thebibliography}
\end{document}